\documentclass[12pt]{amsart}
\usepackage{fancyhdr,amsmath,amssymb,latexsym,verbatim,tikz, ulem}
\usepackage[total={6.5in,9.5in},centering,includefoot,includehead]{geometry}
\usepackage{graphicx}
\usepackage{caption}
\usepackage{subcaption}
\usepackage{xcolor}

\usetikzlibrary{arrows,quotes}
\usetikzlibrary{positioning}

\begingroup
\newtheorem{thm}{Theorem}[section]

\newtheorem{lemma}[thm]{Lemma}       
\newtheorem{prop}[thm]{Proposition} 
\newtheorem{defn}[thm]{Definition}
\newtheorem{ex}[thm]{Example}
\newtheorem{cor}[thm]{Corollary}

\newtheorem{remark}[thm]{Remark}  
\newtheorem{notation}[thm]{Notation}
\endgroup

\newcommand{\E}{{\mathcal{A}}}

\newcommand{\w}{{\mathfrak{w}}}

\newcommand{\R}{{\mathbb{R}}}

\newcommand{\n}{{\mathfrak{n}}}

\newcommand{\V}{{\mathfrak{v}}}

\newcommand{\spn}{{\mbox{span}}}

\newcommand{\ds}{\displaystyle}

\newcommand{\Span}{\operatorname{span}}

\newcommand{\ab}{\mathfrak{a}}
\title[]{Abelian factors in 2-step nilpotent Lie algebras constructed from graphs}

\author[]{Rachelle DeCoste}
\address{Rachelle DeCoste, Wheaton College, Norton, MA}
\email{decoste\_rachelle@wheatoncollege.edu}

\author[]{Lisa DeMeyer}
\address{Lisa DeMeyer, Central Michigan University, Mount Pleasant, MI}
\email{demey1la@cmich.edu}

\author[]{Meera Mainkar}
\address{Meera Mainkar, Central Michigan University, Mount Pleasant, MI}
\email{maink1m@cmich.edu}

\author[]{Allie Ray}
\address{Allie Ray, Birmingham-Southern College, Birmingham, AL}
\email{adray@bsc.edu}

\date{\today}
\begin{document}

\pagestyle{fancy} 

\begin{abstract}
We consider real 2-step metric nilpotent Lie algebras associated to graphs with possibly repeated edge labels as constructed by Ray in 2016.
We determine how the structure of the edge labeling within the graph contributes to the abelian factor in these Lie algebras. Furthermore, we explicitly compute the abelian factor of the 2-step nilpotent Lie algebras associated with some families of graphs such as star graphs, cycles, Schreier graphs, and properly edge-colored graphs. We also study the singularity properties of these  Lie algebras in certain cases.

\end{abstract}

\maketitle 

\section{\textbf{Background}}

 A 2-step nilpotent Lie algebra $\n$ is a Lie algebra where each Lie bracket $[X, [Y, Z]]=0$ for all elements  $X, Y, Z$ of $\n$. Investigating a 2-step nilpotent Lie algebra with inner product has become the main tool to study the associated 2-step nilpotent Lie group with left invariant metric.  This research was initiated by Eberlein \cite{E, E2} and Kaplan \cite{K1,K2}.
Recently, there has been a lot of interest in investigating the structure of 2-step nilpotent Lie algebras associated with various type of graphs. Due to the associated combinatorial structure, these 2-step nilpotent Lie algebras have proven to be an interesting source of examples in the study of geometry and dynamics of nilpotent Lie groups, nilmanifolds, infra-nilmanifolds, solvmanifolds, etc. See, for example \cite{AAA, AD, CMS, CDF, DM, DDM, DeMa, F, FJ, GGI, LW2, M, N, O, PS, PT, R}.

In this paper, we study the structure of real 2-step nilpotent Lie algebras constructed from connected directed edge-labeled graphs. A first construction on simple graphs with no repeated edge labels was introduced by S. Dani and the third author in \cite{DM} in order to study Anosov diffeomorphisms of nilmanifolds. In \cite{DDM}, the first three authors further investigated the geometry of these constructed Lie algebras with a specific metric. More specifically, the authors studied the properties of singularity, Heisenberg-like Lie algebras and closed geodesic density properties on the associated nilmanifolds. 

The construction from \cite{DM} was generalized by the fourth author in \cite{R}.  They defined a new 2-step nilpotent Lie algebra construction associated with graphs (simple or non-simple) with possibly repeated edge labels. These Lie algebras can be thought of as a quotient of the 2-step nilpotent Lie algebras associated with simple graphs as in \cite{DM}. They also studied the Lie algebras constructed from Schreier graphs, which are a generalization of Cayley graphs. This generalized construction was also used in \cite{PS} where the authors study Lie algebras associated with uniform graphs.  

One issue that arises with the construction of 2-step nilpotent Lie algebras from graphs with repeated edge labels is the presence of an abelian factor. By definition, the center of any 2-step nilpotent Lie algebra contains its derived algebra. For the Lie algebras constructed from simple graphs with non-repeated edge labels, the center is equal to its derived algebra. However, if we allow repeated edge labels, this is not always the case and the size of the center compared with the size of its derived algebra depends on the structure of edge labels within the associated graph. The vector space complement of the derived algebra in the center is called an  {\em  abelian factor} of the 2-step nilpotent Lie algebra.    

In this paper, we consider certain families of graphs with repeated edge labels and then determine the existence or absence of abelian factors in the associated 2-step nilpotent Lie algebras. Furthermore, we explicitly describe their non-trivial abelian factors in terms of the associated graphs. We also find a basis for the orthogonal complement of the center with respect to the  natural inner product on the 2-step nilpotent Lie algebra arising from certain families of graphs.

This paper is organized as follows. In Section~\ref{Preliminaries}, we introduce and recall the basic terminology and notation in graph theory and Lie algebras. We also give a sufficient condition in terms of a graph for the associated 2-step nilpotent Lie algebra to have a trivial abelian factor.   We then consider the families of star graphs, cycles, and Schreier graphs in Sections~\ref{Sec:stars}-\ref{Sec:Schreier}, respectively. Section~\ref{Sec:coloring} includes some applications related to graph coloring. We end with a discussion of open questions and further research directions, and we would like to thank the referee for the suggestion of adding this conclusion section.

\section{\textbf{Preliminaries}}\label{Preliminaries}

Let $G=(V,E,C)$ denote a connected, directed graph with vertex set $V(G)$, edge set $E(G)$, and edge labeling given by a surjective map $c:E(G)\rightarrow C(G)=\{Z_1, Z_2, \ldots, Z_p\}$. Note that we will suppress $G$ if it is understood in context. If the direction of an edge from a vertex $v$ to a vertex $w$ is relevant, we denote the edge by $(v, w)$. Otherwise we denote it by $vw$. 

\begin{defn}\label{defn:construction} (\cite{R} Remark 3.2-3.3) Given a directed edge-labeled graph $G=(V,E,C)$, we define a metric 2-step nilpotent Lie algebra $\n_G$ in the following way. First, define the real vector space $\n_G=\V\oplus\w$ where $\V$ is the span of the vertex set $V(G)$ and $\w$ is the span of the edge labels $C(G)$. The metric is defined so that $V \cup C$ is an orthonormal basis for $\n_G$. Define the Lie bracket on $\n_G$ by:
$$[v_i,v_j]=\sum_{\ell=1}^{|C|} (\epsilon_{\ell}-\epsilon'_{\ell})Z_{\ell}$$
where
 \begin{equation*} \epsilon_{\ell}= \begin{cases}
        1 & \text{if } c(v_i,v_j)=Z_{\ell} \\
        0 & otherwise,
    \end{cases}
\end{equation*}
and
 \begin{equation*} \epsilon'_{\ell}= \begin{cases}
        1 & \text{if } c(v_j,v_i)=Z_{\ell} \\
        0 & otherwise.
    \end{cases}
\end{equation*}
All other brackets not defined by linearity or skew symmetry are zero. 
\end{defn}

\begin{remark} Note that with this definition, the derived algebra $[\n_G,\n_G]=\w$.  Also note that this Lie algebra will be 2-step nilpotent as $[\n_G,\n_G]\subseteq \spn\{C(G)\}$ and thus $[\n_G,[\n_G,\n_G]]=0$.
\end{remark}

\begin{ex} For the graph in Figure~\ref{Fig:1}, the orthonormal basis for the Lie algebra $\n$ using Definition ~\ref{defn:construction} would be $\{v_1,v_2,v_3,Z_1,Z_2\}$ with brackets $[v_1,v_2]=Z_1$, $[v_1,v_3]=Z_2$, and $[v_2,v_3]=Z_1-Z_2$. All other brackets not defined by linearity or skew symmetry are zero.

\begin{figure}[ht]
\begin{center}
\begin{tikzpicture}[->,>=stealth',shorten >=1pt,auto,node distance=3cm,
  thick,main node/.style={circle,fill,draw,scale=.5,font=\sffamily\Large\bfseries}]
  
\node[main node] (v3) [label=left:{$v_3$}] at (0,0){};
\node[main node] (v1) [label=above:{$v_1$}] at (1,2){};
\node[main node] (v2) [label=right:{$v_2$}] at (2,0) {};
		
\path[every node/.style={font=\sffamily\small}]
(v1) edge node [right] {$Z_1$}  (v2)
(v2) edge node [below] {$Z_1$}  (v3)
(v3) edge  [bend left,"$Z_2$"]	 (v2)
(v1) edge node  [left] {$Z_2$}  (v3);
\end{tikzpicture}
\caption{Directed graph on 3 vertices}
\label{Fig:1}
\end{center}
\end{figure}
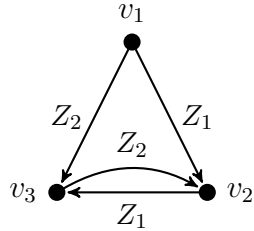
\end{ex}

To study these 2-step nilpotent metric Lie algebras, we follow the approach in \cite{E,K1}. Let $\n$ denote a real 2-step nilpotent Lie algebra and $[\n, \n]$ denote its derived algebra. Let $\V$ denote a linear direct summand of $[\n, \n]$ in $\n$, i.e. $\n = \V\oplus  [\n, \n]$. Choose an inner product $\langle \, , \,\rangle$ on $\n$ such that a basis of $\V$ together with a basis of $[\n, \n]$ forms an orthonormal basis of $\n$. For  $Z \in [\n, \n]$, the skew-symmetric operator $j(Z)$ on $\V$ is defined such that for each $X \in \V$, $j(Z)(X)$ is the unique element satisfying
\begin{eqnarray}\label{jz} \langle j(Z)(X), Y \rangle = \langle Z, [X, Y] \rangle\,\, \text{ for all } Y \in \V. \end{eqnarray}
Note that some authors, such as \cite{E,K1}, have defined this operator based on the center, $\mathcal{Z}(\n)$. In this case, $\n=\mathcal{V}\oplus\mathcal{Z}(\n)$ where $\mathcal{V}=\mathcal{Z}(\n)^{\perp}$. We instead choose to use the derived algebra for ease of calculations since the edge labels form a basis for the derived algebra in our construction.

Let $\n=\V\oplus[\n,\n]$ as above. Since $\n$ is 2-step nilpotent, $[\n, \n]$ is contained in the center $\mathcal{Z}(\n)$. 
Then there exists a unique subspace $\ab$ of $\V$ such that $\mathcal{Z}(\n) = [\n, \n] \oplus \ab$. This subspace $\ab$ is known as a \emph{maximal abelian factor of $\n$} (see \cite{LW1} for example). We note that $\ab$ depends on a choice of the direct summand $\V$. Hence, $\n=\mathcal{Z}(\n)^{\perp}\oplus\ab\oplus[\n,\n]$ and $\V=[\n,\n]^{\perp}=\mathcal{Z}(\n)^{\perp}\oplus\ab$.

\begin{defn}\label{defn:abfactor}
If $\n_G$ is a 2-step nilpotent Lie algebra associated with a directed edge-labeled graph $G = (V, E, C)$ by Definition \ref{defn:construction}, then we have a canonical choice for $\V$, i.e. $\V$ is the span of the vertex set $V$. Therefore the corresponding abelian factor $\ab$ is canonically determined to be the unique subspace of $\V$ such that the center $\mathcal{Z}(\n) = [\n_G, \n_G] \oplus \ab$. We call $\ab$ the \emph{abelian factor of $\n_G$}. If $\ab$ is trivial, we will say $\n_G$ has \emph{no abelian factor}.
\end{defn} 

The abelian factor is a Lie algebraic notion and is not metric dependent. However, as shown in the Lemma below, it is often useful to compute the abelian factor using a specific metric.

\begin{lemma}\label{Lemma:AbFactor}
Let $G= (V, E, C)$ denote a directed, edge-labeled graph. Then the abelian factor $\ab$ of $\n_G$ is equal to ${\ds \bigcap_{Z \in C}^{}\ker (j(Z))}.$
\end{lemma}

\begin{proof}
Let $X$ be an element of $\ab$ and let $Z \in C$. Then for all $Y \in \V$, we  have  $[X, Y] = 0$  and hence $\langle Z, [X, Y] \rangle = 0$.  This implies that $\langle j(Z)(X), Y \rangle = 0$ for all $Y \in \V$. Therefore $j(Z)(X) = 0$ and hence $\ab \subseteq {\ds \bigcap_{Z \in C}^{}\ker (j(Z))}.$ 

Now assume that  $X' \in \V$ and $j(Z)(X') = 0 $ for each $Z \in C$. Let $Y \in \V$ be arbitrary. Then $\langle j(Z)(X'), Y \rangle = 0$ for all  $Z\in C$. Hence $\langle Z, [X', Y] \rangle = 0$ for all $Z \in C$. Since $C$ is an orthonormal basis for $[ \n_G, \n_G]$, we have $[X', Y] = 0$.
Therefore $X' \in \ab$ as $Y \in \V$ is arbitrary. We have proved that ${\ds \bigcap_{Z \in C}^{}\ker (j(Z))}\subseteq \ab$.
\end{proof}

When each edge has a unique label, as in the construction in \cite{DM}, the resulting 2-step nilpotent Lie algebra has no abelian factor. This is often not the case if repeated edge labels are allowed. The next results illustrate how common the presence of an abelian factor is in the case when repeated edge labels are allowed and give us some basic tools for which to investigate the presence or lack of an abelian factor in the resulting Lie algebras.

\begin{prop} If $G$ is a directed edge-labeled graph on an odd number of vertices with one edge label, then $\n_G$ has a nontrivial abelian factor. \label{oddvert}
\end{prop}

\begin{proof}
 Let every edge in $G$ be labeled by $Z$ and let $|V(G)|=n$, where $n$ is odd. Since $\V$ is odd dimensional and $j(Z): \V \rightarrow \V$ is a skew symmetric linear map, there exists some nonzero $X\in \V$ with $X\in \ker (j(Z))$.  Since $[\n,\n]$ is 1 dimensional,  $[\n,\n] = \{aZ| a\in \R\}$, so by linearity of $j$, $X$ is in the kernel of $j(Z')$ for any $Z'\in [\n,\n]$.  Hence, $\n_G$ has a nontrivial abelian factor by Lemma~\ref{Lemma:AbFactor}.
\end{proof}

To determine the abelian factor from graph properties, we must look at how edges with the same label are connected to each other. For the rest of this section, we assume that our graph $G$ is simple.
  
The {\em neighborhood of a vertex} $v$, denoted by $N(v)$,  is defined as the set of all vertices in graph $G$ which are adjacent to $v$ without considering edge direction.
\begin{defn}\label{defn:znbhd} For a vertex $v$ and an edge label $Z$, we define a \emph{$Z$-neighborhood} of $v$ as follows: $N_Z(v) = \{ w \in N(x) \, | \, c(vw) = Z  \}.$ \end{defn}
\begin{defn}\label{defn:zdeg} Let $v$ be a vertex in graph $G$. Define the \emph{edge-label degree} of $v$ by $deg_Z (v) =|N_Z(v)|$, ie. the number of edges incident to vertex $v$ with label $Z$. \end{defn}

\begin{defn}\label{defn:scripte} For a graph $G$, we define the following set: $$\E(G)=\{ v \in V(G) \, | \,   \mbox{ every } y\in N(v) \mbox{ has } deg_{c(vy)}(y)>1\}.$$ Note that we will suppress $G$ if it is understood in context.\end{defn}

\begin{ex}\label{exC5} Given the graphs in Figure~\ref{Fig:Ex1}, $\E(G_1)=\{v_1,v_2,v_3,v_4\}$ and $\E(G_2)=\{v_1,v_5\}$.

\begin{figure}[ht]
\begin{minipage}{.45\textwidth}
\begin{center}
\begin{tikzpicture}[->,>=stealth',shorten >=1pt,auto,node distance=3cm,
  thick,main node/.style={circle,fill,draw,scale=.5,font=\sffamily\Large\bfseries}]
  
 \node[main node] (v1) [label=above:{$v_1$}] at (0,2){};
  \node[main node] (v2) [label=above:{$v_2$}] at (2,2){};
  \node[main node] (v3) [label=below:{$v_3$}] at (2,0) {};
	\node[main node] (v4)  [label=below:{$v_4$}] at (0,0) {};
		
  \path[every node/.style={font=\sffamily\small}]
    (v1) edge node [above] {$Z_1$}  (v2)
		(v2) edge node [right] {$Z_1$}  (v3)
		  (v3)   edge node  [below] {$Z_1$}  (v4)
			(v4) edge node [left] {$Z_1$}  (v1);

\end{tikzpicture}
\end{center}
\end{minipage}
\begin{minipage}{.45\textwidth}
\begin{center}
\begin{tikzpicture}[->,>=stealth',shorten >=1pt,auto,node distance=3cm,
  thick,main node/.style={circle,fill,scale=.5,draw,font=\sffamily\Large\bfseries}]

  \node[main node] (v3) [label=below:{$v_3$}] at (2,2){};
	\node[main node] (v1)  [label=left:{$v_{1}$}] at(0,3) {};
		\node[main node] (v2)  [label=left:{$v_{2}$}] at (0,1) {};
		\node[main node] (v4)   [label=below:{$v_4$}] at (4,2) {};
		\node[main node] (v5)  [label=right:{$v_{5}$}] at (6,3) {};
		\node[main node] (v6)  [label=right:{$v_{6}$}] at (6,1) {};
		
  \path[every node/.style={font=\sffamily\small}]
    (v1) edge  node [above] {$Z_1$}  (v3)
		  (v2)   edge  node  [below] {$Z_2$}  (v3)
			(v3) edge   node [above] {$Z_1$}  (v4)
			(v4) edge  node [above] {$Z_1$} (v5)
			(v4) edge  node [below] {$Z_2$} (v6);
\end{tikzpicture}
\end{center}
\end{minipage}
\caption{Graphs $G_1$ and $G_2$, respectively}
\label{Fig:Ex1}
\end{figure}
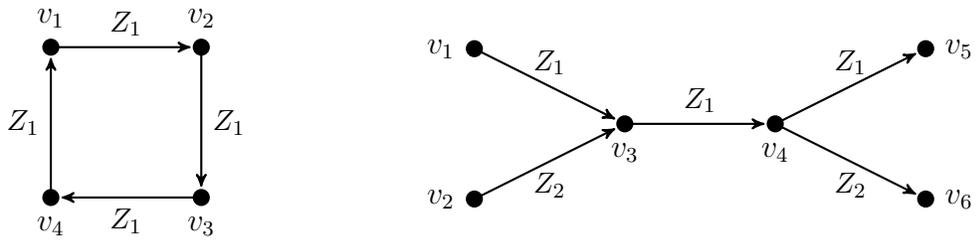
\end{ex}

\begin{thm}\label{Prop:EDRSpan} Let $\n_G$ denote the 2-step nilpotent Lie algebra associated to a simple graph $G$. Then, the abelian factor  $\ab$ of $\n_G$ is contained in the span of $\E$. \end{thm}

\begin{proof} 
Label the vertices in $V(G)$ as $\E=\{v_1,\dots,v_m\}$ and $\E^c=V(G)-\E=\{ v_{m+1}, \ldots, v_n \}$.

For $k=m+1, \ldots, n$, there exists a vertex $y_k$ in $N(v_k)$ such that $deg_{c(v_ky_k)}(y_k) \leq 1$ since $v_k \in \E^c$.

Let $\ds X = \sum_{i=1}^m a_i v_i + \sum_{j=m+1}^n a_j v_j $ be an element of the abelian factor. Hence, we have the bracket $[X, y_k]= 0$ for all $y_k$ obtained in the previous step.

This implies that for each $k=m+1, \ldots, n$:
\begin{align*}
 a_k[v_k, y_k] + \sum_{1\leq \ell\neq k\leq n} a_k[v_{\ell}, y_k] &= 0.
\end{align*}

Recall that $y_k\in \E^c$ means that the label $c(v_ky_k)\neq c(v_{\ell}y_k), \forall \ell \neq k$. Hence $[v_k, y_k] \not\in \Span\{ [v_{\ell}, y_k]\, |\, \ell\neq k\}$. Thus, $[v_k, y_k]$ is linearly independent from $\{[v_{\ell}, y_k]\, |\, \ell\neq k\}$. Hence $a_k = 0$ for all $k=m+1, \ldots, n$.

Thus, $X \in \Span\{ v_{1}, \ldots, v_m\} = \Span \E$.
\end{proof}

The next examples illustrate the necessity and insufficiency of Theorem~\ref{Prop:EDRSpan}.

\begin{ex}\label{Ex:graphEDR} For $G_1$ in Figure~\ref{Fig:Ex1}, a straightforward calculation shows that the associated 2-step nilpotent Lie algebra $\n_{G_1}$ will have a two dimensional abelian factor with basis $\{v_1+v_3,v_2+v_4\}$.
For $G_2$, the associated Lie algebra has no abelian factor even though $\E=\{v_1,v_5\}$. If it did, it would necessarily be contained in the span of $\{v_1, v_5\}$. A straightforward check shows that $\n_{G_2}$ has no abelian factor. This is also verified in Proposition~\ref{Prop:doublestar}.  We note that for this graph, the orientation of the labels $Z_1$ and $Z_2$ does not affect the abelian factor of $\n_{G_2}$. \end{ex}

\begin{ex}\label{Ex:Ex3} Let ${G_1'}$ be the  graph in Figure~\ref{Fig:example2}. Note this is the same graph as $G_1$ in Figure~\ref{Fig:Ex1} except the direction of $v_4v_1$ has switched. We get the same set $\E=\{v_1,v_2,v_3,v_4\}$ as in Example~\ref{exC5} because the definition of $\E$ does not depend on the direction of the edges. However, $\n_{G_1'}$ has no abelian factor. This is verified in part \ref{CorPart:dimA0} of Proposition~\ref{Cor:Cnnonstandardeven}. 

\begin{figure}[ht]
\begin{center}
\begin{tikzpicture}[->,>=stealth',shorten >=1pt,auto,node distance=3cm,
  thick,main node/.style={circle,fill,draw,scale=.5,font=\sffamily\Large\bfseries}]
  
\node[main node] (v1) [label=above:{$v_1$}] at (0,2){};
\node[main node] (v2) [label=above:{$v_2$}] at (2,2){};
\node[main node] (v3) [label=below:{$v_3$}] at (2,0) {};
\node[main node] (v4)  [label=below:{$v_4$}] at (0,0) {};
		
\path[every node/.style={font=\sffamily\small}]
(v1) edge node [above] {$Z_1$}  (v2)
(v2) edge node [right] {$Z_1$}  (v3)
(v3)   edge node  [below] {$Z_1$}  (v4)
(v1) edge node [left] {$Z_1$}  (v4);

\end{tikzpicture}
\caption{Graph $G_{1}'$}
\label{Fig:example2}
\end{center}
\end{figure}
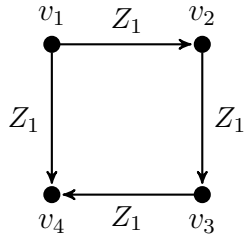
\end{ex}

Sometimes it is advantageous to instead consider what vertices do not contribute to the abelian factor. The next Proposition gives us criteria for that.

\begin{prop}\label{prop:EEdr} Let $\n_G$ denote the 2-step nilpotent Lie algebra associated to a simple graph $G$. The following are equivalent:
\begin{enumerate}
\item\label{Prop1part} The abelian factor of $\n_G$ is contained in the span of $\E$. 
\item\label{Lemma1part} If $\ds \sum_{i=1}^{|V|} a_i v_i$ is in the abelian factor and  $v_j \in N(v_k)$ such that $deg_{c(v_j v_k)}(v_k)=1$, then $a_j=0$.
\end{enumerate}
\end{prop}

\begin{proof}
\noindent  (\ref{Prop1part}) $\Rightarrow$ (\ref{Lemma1part}): If $v_j \in N(v_k)$ such that $deg_{c(v_jv_k)}(v_k)=1$, then by Definition~\ref{defn:scripte}, $v_j \not\in \E$ and therefore $a_j=0$.

\noindent (\ref{Lemma1part}) $\Rightarrow$ (\ref{Prop1part}): Let $\sum a_iv_i$ be in the abelian factor. We need to show if some $a_j \neq 0$ then $v_j\in \E$. Assume $a_j \neq 0$ and $v_j\not\in \E$. By Definition~\ref{defn:scripte}, there is some $v_k \in N(v_j)$ with $deg_{c(v_kv_j)} (v_k) = 1$. Hence assumption (\ref{Lemma1part}) holds and  $a_j=0$, which is a contradiction. Hence, $v_j \in\E$.
\end{proof}

 \begin{cor}\label{cor:nullscriptE} Let $\n_G$ denote the 2-step nilpotent Lie algebra associated to a simple graph $G$. If $\E=\emptyset$, then $\n_G$ has  no abelian factor.  Also, if $\E$ only contains one vertex, then $\n_G$ has no abelian factor.\end{cor}
\begin{proof} If $\E=\emptyset$, the proof follows directly from Proposition~\ref{prop:EEdr}. 
Letting $\E=\{v_i\}$, by Proposition~\ref{prop:EEdr}, the abelian factor will be of the form $a_iv_i$. Then for some $v_k \in N(v_i)$, $[a_iv_i,v_k]=0$, which implies that $a_iZ=0$. Thus, $a_i=0$, and $\n_G$ has no abelian factor.
\end{proof}

Note that when $\E$ contains more than one vertex, there may or may not be an abelian factor. See Example~\ref{Ex:graphEDR}.

\section{\textbf{Star graphs}}\label{Sec:stars}
In the following section, we determine what properties of star graphs lead to an abelian factor in the associated Lie algebra and compute a nice basis for the orthogonal complement of the center.

Let $K_{1,n}$ denote a star graph. The vertex of degree $n$ is called often called the \emph{center vertex} of a star graph in graph theory work. To avoid confusion with the Lie Algebra term, we will call the vertex of degree $n$ in $K_{1,n}$ the \emph{central vertex} and label this vertex $v_0$.  A vertex of degree 1 is called an \emph{end} or a \emph{leaf}.

\begin{notation}\label{notation:star}
For a star graph, $K_{1,n}$, label the central vertex as $v_0$. Let $Z_1, \ldots, Z_k$ denote the edge labels in the graph, and let $m_i$ denote the multiplicity of edge label $Z_i$ for $1\leq i \leq k$.  Label the end vertices in graph $K_{1,n}$ as $v_{i,j}$ so that edges with label $Z_i$ connect the central vertex $v_0$ to vertices $v_{i,j}$ for $1\leq j \leq m_i$ for each $1 \leq i \leq k$.
Denote the direction of each edge via a map  $\delta: \{v_{i,j}\} \rightarrow \{+1, -1\}$ where $\delta(v_{i,j}) = +1$ if the edge $Z_i$ is oriented from $v_0$ to $v_{i,j}$ and $-1$ if edge $Z_i$ is oriented from $v_{i,j}$ to  $v_0$. Hence in the Lie algebra $\n_{K_{1,n}}$, we have $[v_0, v_{i,j}] = \delta(v_{i,j}) Z_i$ for all $1\leq j \leq m_i$ and $1 \leq i \leq k$.\end{notation}

\begin{thm}\label{Thm:DimEdRStar} Let $K_{1,n}$ be a star graph with $k$ distinct edge labels. Then the associated Lie algebra $\n_{K_{1,n}}$ has abelian factor of dimension $\ds n-k$.
\end{thm}

\begin{proof}
Label the edges of $K_{1,n}$ by $Z_1, \ldots, Z_k$ with multiplicities $m_1\geq \cdots \geq m_k$ respectively and let $p$ denote the largest index with $m_p>1$. Then $m_{p+1}, \ldots, m_{k} = 1$ and $m_i>1$ for $1\leq i \leq p$. Observe by Definition~\ref{defn:scripte} that the only vertices in $\E$  are the ends whose incident edge label has multiplicity strictly greater than 1. 
 
Thus, the abelian factor factor is contained in the span of the set $\E = \{v_{i,j} \, | 1\leq i \leq p, 1\leq j\leq m_i\}$.

Let $\ds X = \sum_{i=1}^p \sum_{j=1}^{m_i} a_{i,j}v_{i,j}$. Observe that $[X, v_{i,j}] = 0$ for all $i,j$ since no two  distinct ends $v_{i,j}$ are adjacent and $[v_{i,j}, v_{i,j}]=0$.   So $X$ is in the abelian factor if and only if $\ds 0=[X, v_0]= \sum_{i=1}^p \sum_{j=1}^{m_i} -\delta(v_{i,j} )a_{i,j} Z_i$.  Since $\{Z_i | 1\leq i \leq k\}$ is  a linearly independent set, this occurs precisely when $\ds \sum_{j=1}^{m_i} \delta(v_{i,j}) a_{i,j} = 0$, for each $i$. For each $i$ we have a linear equation with $m_i$ variables; hence, the solution set for each $i$ has dimension $m_i-1$.  Thus, the dimension $D$ of the abelian factor is given by the following computation:

\begin{equation*}
    \begin{split}
     D&= \sum_{i=1}^p (m_i-1)\\
    &=\left(\sum_{i=1}^p m_i\right) -p\\
    &=\sum_{i=1}^p m_i + (k-p)-(k-p) -p\\
    &=\sum_{i=1}^p m_i + \sum_{j=p+1}^k m_j -k+p-p\\
    &=\sum_{i=1}^k m_i -k.
    \end{split}
\end{equation*}

Note that $m_i=1$ for $p+1\leq i \leq k$ so $\sum_{i=p+1}^k m_i = k-p$.

Finally, observe that $\ds\sum_{i=1}^k m_i=n$ since this sum is precisely the number of ends on the star, and thus $D=n-k$. \end{proof}

\begin{cor}\label{Cor:BasisStarGraph}
Let $K_{1,n}$ be a star graph with vertex and edge labeling as in Notation~\ref{notation:star}. Then a basis of the abelian factor in $\n_{K_{1,n}}$ is given by the set \[A=  \bigcup_{i=1}^p \bigcup_{j=1}^{ m_i-1} \left\{\delta(v_{i,m_i}) v_{i,j} - \delta(v_{i, j}) v_{i,m_i}\right\},\]  and an orthonormal basis for the orthogonal complement of the center,  $\mathcal{Z}(\n_{K_{1,n}})^\perp$,  is given by  \[S=\{v_0\}\bigcup_{i=1}^k \left\{\frac{1}{\sqrt{m_i}}\left( \delta(v_{i,1})v_{i,1}+\delta(v_{i,2})v_{i,2}+\ldots+\delta(v_{i,m_i})v_{i,m_i}\right)\right\},\].
\end{cor}   

\begin{proof} A straightforward linear algebra computation to compute the basis for the nullspace of $\ds{\sum_{j=1}^{m_i} \delta(v_{i,j}) a_{i,j}}$ gives that a basis for the abelian factor is given by  $\ds \bigcup_{i=1}^p \bigcup_{j=1}^{ m_i-1} \{\delta(v_{i,m_i}) v_{i,j} - \delta(v_{i, j}) v_{i,m_i}\}$. Note that if $1\leq i \leq p$ then  $m_i>1$, so   $\delta(v_{i,m_i}) v_{i,j} - \delta(v_{i, j}) v_{i,m_i}\neq 0$.

Next, we show that $S$ is a basis of $\mathcal{Z}(\n_{K_{1,n}})^\perp$. First, observe that for each $i$, $$\left\langle\frac{1}{\sqrt{m_i}}\left( \delta(v_{i,1})v_{i,1}+\ldots+\delta(v_{i,j})v_{i,j}+\ldots+\delta(v_{i,m_i})v_{i,m_i}\right),\delta(v_{i,m_i})v_{i,j}-\delta(v_{i,j})v_{i,m_i} \right\rangle$$ $$=\frac{1}{\sqrt{m_i}}\delta(v_{i,j})\delta(v_{i,m_i})-\frac{1}{\sqrt{m_i}}\delta(v_{i,m_i})\delta(v_{i,j})=0,$$ for all $j=1,\dots,m_{i-1}$. Also,  $\langle v_0,\delta(v_{i,j})v_{i,j}-\delta(v_{i,m_i})v_{i,m_i} \rangle=0$ for all $i,j$. This shows that $S\subseteq \mathcal{Z}(\n_{K_{1,n}})^\perp$. Next, it is clear that $S$ is a linearly independent set since the set $V(K_{1,n})$ is linearly independent. Finally, $\dim\mathcal{Z}(\n_{K_{1,n}})^\perp=\dim \n_{K_{1,n}}-\dim\mathcal{Z}(\n_{K_{1,n}})=(1+n+k)-(k+(n-k))=k+1=|S|$. Therefore, $S$ is a basis for $\mathcal{Z}(\n_{K_{1,n}})^{\perp}$.
\end{proof}

\begin{ex}\label{Ex:StarwithDim10LA} The star graph $G$ in Figure~\ref{fig:StarwithDim10LA} is associated to  a 10 dimensional nilpotent Lie Algebra, which has a 3 dimensional abelian factor  with basis $\{v_{1,1}-v_{1,3},v_{1,2}+v_{1,3},-v_{2,1}-v_{2,2}\}$ and a 4 dimensional $\mathcal{Z}(\n_G)^{\perp}$ with orthonormal basis $\{v_0,\frac{1}{\sqrt{3}}(v_{1,1}-v_{1,2}+v_{1,3}),\frac{1}{\sqrt{2}}(v_{2,1}-v_{2,2}),v_{3,1}\}$.

\begin{figure}[ht]
\begin{center}
\begin{tikzpicture}[->,>=stealth',shorten >=1pt,auto,node distance=3cm,
  thick,main node/.style={circle,fill,scale=.5,draw,font=\sffamily\Large\bfseries}]

\node[main node] (v0) [label=above:{$v_0$}] at (2,2){};
\node[main node] (v1) [label=right:{$v_{1,1}$}] at (4,4){};
\node[main node] (v2)  [label=right:{$v_{1,2}$}] at (4,2){};
\node[main node] (v3)  [label=right:{$v_{1,3}$}] at (4,0) {};
\node[main node] (v4)  [label=left:{$v_{2,1}$}] at(0,4) {};
\node[main node] (v5)  [label=left:{$v_{2,2}$}] at (0,2) {};
\node[main node] (v6)  [label=left:{$v_{3,1}$}] at (0,0) {};
					
\path[every node/.style={font=\sffamily\small}]
(v0) edge node [above] {$Z_1$}  (v1)
(v2) edge node [above] {$Z_1$}  (v0)
(v0)   edge node  [above] {$Z_1$}  (v3)
(v0) edge node [above] {$Z_2$}  (v4)
(v5) edge node [above] {$Z_2$} (v0)
(v0) edge node [below] {$Z_3$} (v6);
\end{tikzpicture}
\caption{Example of a star graph}
\label{fig:StarwithDim10LA}
\end{center}
\end{figure}
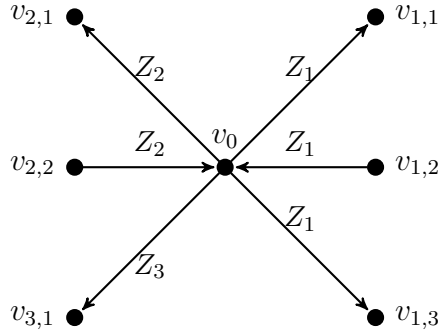
\end{ex}

There are three types of 2-step nilpotent Lie algebras, which depend on the singularity of the $j(Z)$ operator defined in (\ref{jz}).

\begin{defn}\label{defn:singular}
Let $\n=(\mathcal{Z}(\n)^{\perp}\oplus\ab)\oplus [\n,\n]$ be a 2-step nilpotent metric Lie algebra. Then, $\n$ is \emph{singular} if the $j(Z)$ operator restricted to $\mathcal{Z}(\n)^{\perp}$ is singular for all $Z$ in $[\n,\n]$. $\n$ is \emph{nonsingular} if the abelian factor $\ab$ of $\n$ is trivial and the $j(Z)$ operator is nonsingular for all nonzero $Z$ in $[\n,\n]$. By \cite{GM,LP}, if $\n$ is neither singular nor nonsingular, it is \emph{almost nonsingular}.
\end{defn}

Note that $j(Z)X=0$ for all $X\in\ab$. This would result in any 2-step nilpotent Lie algebra with an abelian factor being singular unless we restrict the domain of $j(Z)$ to $\mathcal{Z}(\n)^{\perp}$ as in Definition~\ref{defn:singular}.

We now recall a result of \cite{DDM} where the authors found that if a star graph had unique edge labels, then the resulting Lie algebra was always singular. We then generalize this result when we allow repeated edge labels.

\begin{prop}\label{Prop:StarDDMresult} (\cite{DDM} Ex. 3.4) Let $K_{1,n}$ denote the directed star graph with edge labels $Z_1, \ldots, Z_n$, each with multiplicity 1. Then the resulting Lie algebra $\n_{K_{1,n}}$ is singular. Moreover, if $Z \in [\n_{K_{1,n}},\n_{K_{1,n}}]$ is given by $\sum_{i=1}^n a_iZ_i$, then the eigenvalues of $j(Z)$ are $0, \pm i \sqrt{a_1^2+a_2^2+\cdots + a_n^2} $. 
\end{prop}

\begin{thm}\label{Prop:StarEvals} Let $K_{1,n}$ have edges labeled by $Z_1, \ldots, Z_k$ with multiplicities $m_1, \ldots , m_k$ respectively, and let $Z=\sum_{i=1}^k a_iZ_i$ be an arbitrary element of the derived algebra of $\n_{K_{1,n}}$. Then  the eigenvalues of $j(Z)$ are $\ds 0, \pm i\sqrt{\sum_{\ell=1}^k m_{\ell} a_{\ell}^2}$. \end{thm}

\begin{proof}  We have $j(Z):\V\rightarrow \V$ where $\V=[\n_{K_{1,n}}, \n_{K_{1,n}}]^\perp = \mathcal{Z}(\n_{K_{1,n}})^\perp \oplus \ab$.  The matrix of $j(Z)$ is identically $0$ when restricted to $\ab$. The matrix of $j(Z)$ restricted to $\mathcal{Z}(\n_{K_{1,n}})^{\perp}$ with respect to the orthonormal basis from Corollary \ref{Cor:BasisStarGraph} is given by 
$$\begin{bmatrix} 0 & a_1\sqrt{m_1} & \dots & a_k \sqrt{m_k}\\
-a_1\sqrt{m_1} & 0 & \dots & 0\\
\vdots & \vdots & \ddots & \vdots\\
-a_k \sqrt{m_k} & 0 & \dots & 0
\end{bmatrix}.$$
The characteristic polynomial of this matrix is $(-\lambda)^{k-1}(\lambda^2+\sum_{\ell=1}^{k}m_{\ell}a_{\ell}^2$).
\end{proof}

\begin{cor} Let $K_{1,n}$ be a star graph with possibly repeated edge labels. Then $\n_{K_{1,n}}$ is singular. 
\end{cor}

\begin{proof} The proof follows directly from the proof of Theorem~\ref{Prop:StarEvals} and Definition~\ref{defn:singular}.
\end{proof}
 
We now show that the Lie algebra associated to a star graph with repeated edge labels has a Lie subalgebra with no abelian factor which corresponds to a different star graph with unique weighted edge labels.

Let $\n_G= (\mathcal{Z}(\n_G)^\perp\oplus\ab)\oplus [\n_G, \n_G]$ be a 2-step nilpotent Lie algebra associated to a star graph $G$. Then define $\n_G^*$ to be the Lie algebra given by $\n_G^* = \ab\backslash\n_G$. Thus $\n_G^* \cong \mathcal{Z}(\n_G)^\perp\oplus[\n_G, \n_G]$, and $\n_G^*$ is a Lie algebra which has no abelian factor. Using a slightly modified construction, the Lie algebra $\n_G^*$ can be constructed from graph $ K_{1,k}$ where $k$ is the  number of distinct edge labels in the original graph $G$. In the new graph, all $k$ edges have distinct edge labels $Z_1, \ldots, Z_k$, where each $|Z_i|=1$, and each edge is weighted according to the square root of its multiplicity in the original graph $G$. 

\begin{ex}\label{ex:StarG1} Let $G$ be the graph given in Example \ref{Ex:StarwithDim10LA} and let $\n_G$ denote the resulting Lie algebra. Let $\n_G^* = \mathcal{Z}(\n_G)^\perp\oplus [\n_G, \n_G]$, which  is the Lie algebra resulting from the modified construction described above using directed graph with edge labels of multiplicity one and where each edge label is weighted according to the square root of its multiplicity in graph $G$, as shown in Figure~\ref{Fig:StarG1}
.  The vectors $\{Z_1, Z_2, Z_3\}$ form an orthonormal basis for the derived algebra of $\n_G^*$, which in this case is equal to the center of the Lie algebra in $\n_G^*$. The vectors $\{X_0, X_1, X_2, X_3\}$ are given by the basis for $\mathcal{Z}(\n_G)^\perp$ in Corollary \ref{Cor:BasisStarGraph}. So in this example, $X_0=v_0$, $X_1 = \frac{1}{\sqrt{3}} (v_{1,1} - v_{1,2}+v_{1,3})$, $X_2 = \frac{1}{\sqrt{2}} ( v_{2,1}- v_{2,2})$, and $X_3 = v_{1,3}$.

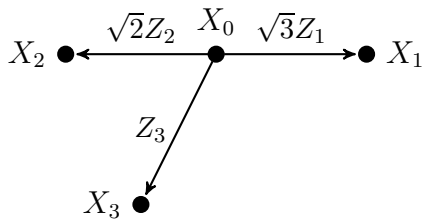
\begin{figure}[ht]
\begin{center}
\begin{tikzpicture}[->,>=stealth',shorten >=1pt,auto,node distance=3cm,
  thick,main node/.style={circle,fill,scale=.5,draw,font=\sffamily\Large\bfseries}]

\node[main node] (X0) [label=above:{$X_0$}] at (2,2){};
\node[main node] (X1)  [label=right:{$X_1$}] at (4,2){};
\node[main node] (X2)  [label=left:{$X_2$}] at (0,2) {};
\node[main node] (X3)  [label=left:{$X_3$}] at (1,0) {};
			
\path[every node/.style={font=\sffamily\small}]
(X0) edge node [above] {$\sqrt{3}Z_1$}  (X1)
(X0) edge node [above] {$\sqrt{2}Z_2$}  (X2)
(X0)   edge node  [left] {$Z_3$}  (X3);

\end{tikzpicture}
\caption{Graph for modified construction used  in Example~\ref{ex:StarG1}}
\label{Fig:StarG1}
\end{center}
\end{figure}
\end{ex}

Next we consider a graph $G$ which is the union of two star graphs whose central vertices are connected by a single edge. 

\begin{ex} The graph in Figure~\ref{Fig:doublestar} is associated to a nilpotent Lie algebra which has 3 dimensional abelian factor  with basis $\{v_{1,1}-v_{1,2},v_{2,1}-v_{2,2},w_{1,1}-w_{1,2}\}$. 
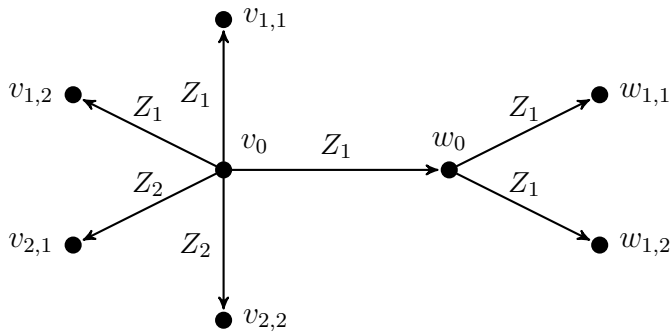
\begin{figure}[ht]
\begin{center}
\begin{tikzpicture}[->,>=stealth',shorten >=1pt,auto,node distance=3cm,
  thick,main node/.style={circle,fill,scale=.5,draw,font=\sffamily\Large\bfseries}]

\node[main node] (v0) [label=above right:{$v_0$}] at (2,2){};
\node[main node] (v1) [label=right:{$v_{1,1}$}] at (2,4){};
\node[main node] (v3)  [label=right:{$v_{2,2}$}] at (2,0) {};
\node[main node] (v4)  [label=left:{$v_{1,2}$}] at(0,3) {};
\node[main node] (v5)  [label=left:{$v_{2,1}$}] at (0,1) {};
\node[main node] (w0)  [label=above:{$w_{0}$}] at (5,2) {};
\node[main node] (v6)  [label=right:{$w_{1,1}$}] at (7,3) {};
\node[main node] (v7)  [label=right:{$w_{1,2}$}] at (7,1) {};
		
  \path[every node/.style={font=\sffamily\small}]
    (v0) edge node [left] {$Z_1$}  (v1)
	(v0)   edge node  [left] {$Z_2$}  (v3)
	(v0) edge node [above] {$Z_1$}  (v4)
	(v0) edge node [above] {$Z_2$} (v5)
	(v0) edge node [above] {$Z_1$} (w0)
	(w0) edge node [above] {$Z_1$} (v6)
	(w0) edge node [above] {$Z_1$} (v7);
\end{tikzpicture}
\caption{Union of two star graphs with adjacent central vertices}
\label{Fig:doublestar}
\end{center}
\end{figure}
\end{ex}

\begin{prop}\label{Prop:doublestar} Let $G$ be the graph which is constructed from two directed, edge-labeled, star graphs $G_1=K_{1,n}$ and $G_2 = K_{1,m}$ where the central vertex of each $G_1$ and $G_2$ is adjacent by a labeled, directed edge.  Let $\ab_i$ denote the abelian factor of $\n_{G_i}$ for $i=1,2$. Then the abelian factor $\ab$ of $\n_G$ is given by $\spn(\ab_1 \cup \ab_2)$. 
\end{prop}

\begin{proof}
Let $G$ be a graph constructed from two star graphs, as described, with vertices labeled as shown in Figure~\ref{Fig:prop8}. Edge labels and edge directions not pictured.

\begin{figure}[ht]
\begin{center}
\begin{tikzpicture}[-,>=stealth',shorten >=1pt,auto,node distance=3cm,
  thick,main node/.style={circle,fill,scale=.5,draw,font=\sffamily\Large\bfseries}]

\node[main node] (v0) [label=above right:{$v_0$}] at (2,2){};
\node[main node] (v1) [label=right:{$v_1$}] at (2,4){};
\node[main node] (v3)  [label=right:{$v_n$}] at (2,0) {};
\node[main node] (v4)  [label=left:{$v_2$}] at(0,3) {};
\node[draw=none,fill=none] (v5)   at (0,1) {$\ddots$};
\node[main node] (w0)  [label=above:{$w_{0}$}] at (5,2) {};
\node[main node] (v6)  [label=right:{$w_1$}] at (7,3) {};
\node[draw=none, fill=none] (v8) [label=right:{$\vdots$}] at (7,2) {};
\node[main node] (v7)  [label=right:{$w_m$}] at (7,1) {};
		
\path[every node/.style={font=\sffamily\small}]
(v0) edge  (v1)
(v0)   edge   (v3)
(v0) edge   (v4)
(v0) edge     (w0)
(w0) edge     (v6)
(w0) edge  (v7);
\end{tikzpicture}
\caption{Vertex labeling for the proof of Proposition \ref{Prop:doublestar}}
\label{Fig:prop8}
\end{center}
\end{figure}
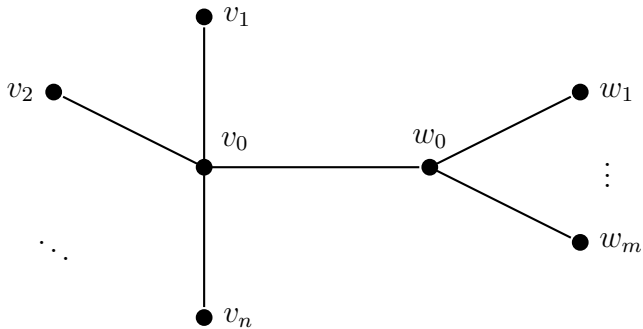

Let $X \in \ab_1 \cup \ab_2$. Assume that $X\in \ab_1$ and write $X= \sum_{i=0}^n a_iv_i$. By Corollary~\ref{Cor:BasisStarGraph}, $a_0=0$. Since $X$ is in $\ab_1$, we have $[X,v_i]=0$ for all $v_i$ in $G_1$. It is also true that $[X,w_j]=0$ for all $w_j$ in $G_2$ because there are no adjacent edges between $v_1, \ldots, v_n$ and vertices in $G_2$. A similar result holds if $X \in \ab_2$. Therefore $X \in \ab$. Since $\ab$ is a vector space containing $\ab_1\cup \ab_2$, we have $\spn(\ab_1 \cup \ab_2) \subseteq \ab$.

To show equality, suppose $\ds X = \sum_{i=0}^n a_i v_i + \sum_{j=0}^m b_j w_j \in \ab$.
For $1\leq i \leq n$, $0= [X, v_i] = [a_0 v_0, v_i]= a_0 [v_0,v_i]$, hence $a_0=0$. Similarly, for $1\leq j \leq m$, $0=[X, w_j] = [b_0w_0, w_j] = b_0[w_0, w_j]$, and hence $b_0=0$. Therefore, $\ds X = \sum_{i=1}^n a_i v_i + \sum_{j=1}^m b_j w_j$.

We will  now show that $\ds
\sum_{i=1}^n a_i v_i \in \ab_1$ and $\ds \sum_{j=1}^m
b_j w_j \in \ab_2$. Since $X \in \ab$,  we have $[X, v_0]=0$. 
Hence, $ [\sum_{i=1}^n a_iv_i, v_0] = 0$ as $[w_j, v_0] = 0$ for all $j = 1, \ldots, m$. Also $[\sum_{i=1}^n a_iv_i, v_k] = 0$ for all $k = 1, \ldots, n$ and therefore $\sum_{i=1}^n a_iv_i \in \ab_1$. Similarly, we can show that $\sum_{j=1}^m b_j w_j \in \ab_2$. Thus, $\spn(\ab_1 \cup \ab_2) = \ab$. 
\end{proof}

\begin{cor}\label{Cor:dimEdRDoubleStar} Let $G_1=K_{1,n}$ be a directed, edge-labeled star graph with $k_1$ edge labels and $G_2=K_{1,m}$ a directed edge-labeled star graph with $k_2$ edge labels. Let $G$ be the graph constructed from $G_1$ and $G_2$ where the central vertex of each $G_1$ and $G_2$ is adjacent by a labeled, directed edge. Then the dimension of the abelian factor of $G$ is $(n+m)-(k_1 + k_2)$.  
\end{cor}

\begin{proof}
The dimension of the span of $E_1 \cup E_2$ is given by the sum of $dim(E_1)= n-k_1$ and $dim(E_2) = m-k_2$.
\end{proof}

Note that the quantity $k_1+k_2$ in Corollary~\ref{Cor:dimEdRDoubleStar} is not necessarily the number of distinct edge labels in the graph $G$. Edge labels in $G_1$ and $G_2$ are not necessarily distinct and the label of the edge connecting the central vertices is irrelevant for this result. 

Observe that the result on the double star graphs cannot be extended to trees in general. As we see from the next example, there are trees with an internal vertex that is involved in the abelian factor. 

\begin{ex}
Let $\n_G$ be the Lie algebra constructed from the graph $G$ in Figure~\ref{Fig:tree}. In this graph,  $\E=\{w_0,v_1,v_2,u_1,u_2\}$. The corresponding $\n_G$ has a three dimensional abelian factor with basis $\{v_1-v_2, v_2+w_0+u_2, u_1-u_2\}$. 

\begin{figure}[ht]
\begin{center}
\begin{tikzpicture}[->,>=stealth',shorten >=1pt,auto,node distance=3cm,
  thick,main node/.style={circle,fill,scale=.5,draw,font=\sffamily\Large\bfseries}]

\node[main node] (v0) [label=below:{$v_0$}] at (2,2){};
\node[main node] (v1)  [label=left:{$v_{1}$}] at(0,3) {};
\node[main node] (v2)  [label=left:{$v_{2}$}] at (0,1) {};
\node[main node] (w0)  [label=below:{$w_{0}$}] at (4,2) {};
\node[main node] (u0) [label=below:{$u_0$}] at (6,2) {};
\node[main node] (u1)  [label=right:{$u_{1}$}] at (8,3) {};
\node[main node] (u2)  [label=right:{$u_{2}$}] at (8,1) {};
		
\path[every node/.style={font=\sffamily\small}]
(v0) edge  node [above right] {$Z_1$}  (v1)
(v0)   edge  node  [above left] {$Z_1$}  (v2)
(v0) edge   node [above] {$Z_1$}  (w0)
(u0) edge  node [above] {$Z_2$} (w0)
(u0) edge  node [above left] {$Z_2$} (u1)
(u0) edge  node [above right] {$Z_2$} (u2);
\end{tikzpicture}
\caption{Example of a tree that is not a star graph}
\label{Fig:tree}
\end{center}
\end{figure}
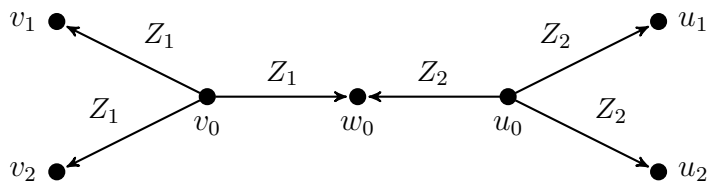
\end{ex}

\section{\textbf{Cycles}}\label{Sec:cycles}

Let $C_n$ be a cycle with $n$ vertices and $n$ edges, where we label the vertices  $\{v_1,v_2,\dots, v_n\}$, with $v_i$ adjacent to $v_{i+1}$ for $1\leq i\leq n-1$ and $v_n$ adjacent to $v_1$, without specifying the direction of the edges. Denote the edge labels by $Z_j$ for $1\leq j\leq p$ for some positive integer $p$. 

\begin{defn}\label{defn:standorcn} In the cycle $C_n$, define an edge to have \underline{standard orientation} if it is directed as $(v_i,v_{i+1})$ for $1\leq i \leq n-1$ or $(v_n,v_1)$. An edge has \underline{opposite standard orientation} if it is directed as $(v_{i+1},v_i)$ for $1 \leq i \leq n-1$ or $(v_1,v_n)$. We define the graph $C_n$ to have \underline{standard orientation} if all edges have standard orientation.  \end{defn}

The cycle in Example \ref{Ex:Ex3} has three edges with standard orientation and one with opposite standard orientation. The cycle in Example \ref{exC5} has standard orientation. 

\begin{thm}\label{Prop:Cnstandard}  Let $C_n$ be a cycle with standard orientation and all edge labels the same. Let $\ab$ denote the abelian factor of the associated Lie algebra $\n_{C_n}$. 
\begin{enumerate}
\item If $n$ is even, the dimension of $\ab$ is 2 with basis  $\displaystyle \left\{ \sum_{i=1}^{k}v_{2i-1}, \sum_{i=1}^kv_{2i}\right\}$, where $n=2k$.
\item If $n$ is odd, the dimension of $\ab$ is 1 with basis $\displaystyle \left\{ \sum_{i=1}^{n}v_i\right\}$.
\end{enumerate}
\end{thm}

\begin{proof}
Assume $C_n$ has the standard orientation and all edges have the same label, $Z$. Then the brackets in the Lie algebra $\n_{C_n}$ are given by  $[v_i,v_{i+1}]=[v_n,v_1]=Z$ for $i=1,\dots,n-1$, with skew symmetry and linearity, and all other brackets are zero.  Let $X=\sum_{i=1}^n a_iv_i$ be an element in the abelian factor of  $\n_{C_n}$.  Consider $[X,v_i]$ for each $i=1,\dots, n$. By the adjacency  and directions  of edges in the graph $C_n$, we have $[X,v_i]=(a_{i-1}-a_{i+1})Z$ for $i=2,\dots,n-1$,  $[X,v_1]=(a_n-a_2)Z$, and $[X,v_n]=(a_{n-1}-a_1)Z$.  Since $[X, v_i] = 0$ for all $i = 1, \ldots, n$, we have the following coefficient relationships: $a_{i-1}=a_{i+1}$ for $i=2,\dots, n-1$, $a_2=a_n$ and $a_1=a_{n-1}$.

If $n$ is even, the above computation shows that the abelian factor is spanned by two vectors, one is the sum of the even indexed vertices and the other is the sum of the odd indexed vertices. This results in a 2-dimensional abelian factor with the basis given above. 

If $n$ is odd, then  using  $a_2=a_n$, we obtain that each even indexed term is equal to each odd indexed term. Therefore, all $a_i$ are equal.  Thus there is a 1-dimensional abelian factor with the basis given above.
\end{proof}

\begin{prop}\label{Cor:Cnnonstandardeven}
Let $C_n$ be a cycle with all edge labels the same.  Let $m$ be the number of edges which have opposite standard orientation.  Let $\ab$ denote the abelian factor of the associated Lie algebra $\n_{C_n}$.
\begin{enumerate}
    \item\label{CorPart:dimA0} If $n=2k$ is even and $m$ is odd, the dimension of $\ab$ is zero. 
    \item If $n=2k$ is even and $m$ is even, the dimension of $\ab$ is 2.  The basis is \[\displaystyle \left\{ \sum_{i=1}^{k}\epsilon_{2i-1}v_{2i-1}, \sum_{i=1}^k\epsilon_{2i}v_{2i}\right\},\] where $\epsilon_j=\pm 1$.
    \item If $n=2k+1$ is odd, the dimension of $\ab$ is 1 with basis $\displaystyle \left\{ \sum_{i=1}^{n}\epsilon_iv_i\right\}$ where $\epsilon_i=\pm 1$. 
\end{enumerate}
\end{prop}

\begin{proof} For $1 \leq i \leq n-1,$ define $\delta_i=1$ if the edge $v_iv_{i+1}$ has standard orientation and $-1$ if the edge has opposite standard orientation. Similarly, define $\delta_n=1$ if the edges $v_nv_1$ has standard orientation and $-1$ if the edge has opposite standard orientation. In this more general case, the bracket relations are given by $[v_i, v_{i+1}] = \delta_i Z$ for $1\leq i \leq n-1$ and $[v_{n}, v_1]= \delta_n Z$. Let $X=\sum_{i=1}^n a_iv_i$ be an element in the abelian factor of $\n_{C_n}$. Following the same calculations as in the proof of Theorem~\ref{Prop:Cnstandard}, we get the following relations between the coefficients of $X$: $a_i=\delta_i\delta_{i+1}a_{i+2}$ for $1 \leq i \leq n-2$, $a_{n-1}=\delta_{n-1}\delta_n a_1$, and $a_n=\delta_n\delta_1 a_2$. 

Note that for every change in orientation of an edge, two of these relations will change signs, one for coefficients with even index and one for coefficients with odd index. In the case where $n$ is even and $m$ is odd, this means that an odd number of coefficient relations will change signs in both the odd indexed coefficients and even indexed coefficients. This will result in the equations $a_2=\pm a_4=\ldots=\pm a_n=-a_2$ and $a_1=\pm a_3=\ldots=\pm a_{n-1}=-a_1$, which means $a_i=0$ for all $1\leq i \leq n$ and hence there is no abelian factor. In the other two cases, these relations give us the basis given in the corollary. 
\end{proof}

Example \ref{Ex:Ex3} illustrates case \ref{CorPart:dimA0} in Proposition~\ref{Cor:Cnnonstandardeven}. 

In the next results, we investigate the case where $C_n$ has at least two distinct edge labels. We will need the following terminology.

A {\em path} is a sequence of distinct adjacent edges within a graph, disregarding the direction of the edges. The {\em length of a path} is equal to the number of edges in the path. 
\begin{defn}\label{defn:samelabelpath} A \emph{same-labeled path} is a path where the label on each edge is the same and there are exactly two vertices in the path of edge-label degree 1 for that label, i.e. the path is of maximal length.
\end{defn}

\begin{prop} If $C_n$ is a cycle with at least two distinct edge labels and at most one same-labeled path on 2 or more edges, then $\n_{C_n}$ has no abelian factor. \label{cyclemorecolors} \end{prop}

\begin{proof}
Let $v_1 - v_2 - \cdots - v_{k+1}$ be the only  same-labeled path in $C_n$ of length 2 or more.

Note that if $k<4$, the path has length 3 or less, and then by Definition~\ref{defn:scripte}, $\E=\emptyset$, hence by Corollary~\ref{cor:nullscriptE}, there is no abelian factor.

If $k\geq 4$ then  $\E=\{v_3, v_4, \ldots , v_{k-1}\}$. Solving $$0= \left[a_3v_3+a_4v_4+\ldots+a_{k-1}v_{k-1},v_{\ell}\right]$$ for each $\ell=1,2,\ldots,k+1$ yields $a_i=0$ for all $3\leq i \leq k-1$. Therefore, by Proposition~\ref{prop:EEdr}, $\n_{C_n}$ has no abelian factor.
\end{proof}

\begin{lemma}\label{lem:everyother} Let $P_n$ be the path on $n\geq 3$ vertices labeled $v_1, \ldots, v_n$ with directed edges $(v_i,v_{i+1})$ for $i=1,\ldots, n-1$ and all having the same edge label $Z$. Then an element of the abelian factor $X=\sum_{i=1}^n a_iv_i$ of the associated Lie algebra $\n_{P_n}$ will have $a_{i-1}=a_{i+1}$, for $2 \leq i \leq n-1$.
\end{lemma}
\begin{proof}
For $2 \leq i\leq n-1$, $[X,v_i]=0$ since $X$ is in the abelian factor.  Then since $N(v_i) = \{v_{i-1}, v_{i+1}\}$ and all edges have label $Z$, we have  $[X, v_i] = a_{i-1}Z-a_{i+1}Z=0$. Therefore,  $a_{i-1}=a_{i+1}$ for $2\leq i \leq n-1$.
\end{proof}

\begin{prop}\label{prop:Cnoddpath} Let $C_n$ be a cycle with standard orientation and at least two distinct edge labels. The Lie algebra $\n_{C_n}$ has a nontrivial abelian factor if and only if the induced subgraph of $C_n$ on each edge label consists of the disjoint union of paths each with an even edge-length. 
\end{prop}

\begin{proof}  Let $p$ denote the number of edge labels. Since there are two or more edge labels, $n\geq p\geq 2$.  Since $C_n$ is a cycle, the induced subgraph on any edge label consists of the finite disjoint union of  paths.

$(\Rightarrow $) Assume that $\n_{C_n}$ has a nontrivial abelian factor and suppose that for some edge label $Z_0$, the induced subgraph on that edge label includes a connected component that is a path on an odd number of edges. Denote this odd length path by $v_1 - v_2 -\cdots - v_k$. Since there are an odd number of edges in this path, the index  $k$ is even. Suppose $X=\sum_{i=1}^k a_i v_i +Y$  is in the abelian factor of $\n_{C_n}$ where the expansion of $Y$ in terms of the vertex labels will have no terms from the odd path made up by the vertices $v_i, 1\leq i\leq k$.  By Proposition \ref{prop:EEdr} Part (\ref{Lemma1part}), $a_2=a_{k-1}=0$, and by Lemma~\ref{lem:everyother},  we have $a_1=a_3=\cdots = a_{k-1}$ and $a_2=a_4=\cdots = a_{k}$. Since $a_2=0$, we have $a_2=a_4 = \cdots = a_k = 0$, and $a_{k-1}=0$ implies $a_1=a_3=\cdots = a_{k-1}=0$. Thus, any element $X$ in the abelian factor has no nonzero terms from vertices in any odd path. 

Then for a same-labeled path in $C_n$ incident to an odd path, we can again use Proposition \ref{prop:EEdr} Part (\ref{Lemma1part}) and Lemma \ref{lem:everyother} in the same manner as above. In addition, we have the fact that the coefficient of the term corresponding to the vertex where the two paths are incident is zero. Thus, the terms of an element in the abelian factor will have coefficient zero for any term associated with vertex labels from that path. Since $C_n$ is decomposed into finitely many paths, each incident to another path at both the initial and final vertex, we get that any element $X$ in the abelian factor is zero.   Hence $\n_{C_n}$ has no abelian factor. 

($\Leftarrow$) Assume the induced subgraph of $C_n$ on each edge label consists of the finite disjoint union of paths each with an even edge-length. This condition implies that $n$ is even. Decompose $C_n$ into these paths, where each subgraph is a same-labeled path and incident paths are of distinct edge labels.  Denote the vertices of $C_n$ by $\{v_1, \ldots, v_n\}$ and denote the edge labels of each same-labeled path by $Z_1, \ldots, Z_p$ where $Z_1$ is the edge label of $(v_1, v_2)$ and the edge labels are indexed in the order they appear in standard orientation. Note that not all $Z_i$'s are necessarily distinct, but because we are decomposing on induced graphs for each edge label, we do assume that $Z_{\ell}\neq Z_{\ell+1}$ for $1\leq \ell \leq p-1$ and $Z_1\neq Z_p$. Then a straightforward computation shows that $X = \sum_{i=1}^{n/2} v_{2i-1}$ is in the abelian factor. To see this, first observe that $[X, v_{2j-1}] = 0$ for each $1\leq j\leq n/2$ since $[v_{2j-1}, v_{2j-1}]=0$ and there are no adjacency relations among distinct odd indexed vertices so $[v_{2i-1}, v_{2j-1}]=0$ for $i\neq j$. Observe that the edge labels of the two edges incident to an even indexed vertex are the same since each same-labeled path is of even length and starts at a vertex with odd index.  Thus, $[v_{2i-1}, v_{2i}]= [ v_{2i}, v_{2i+1}]$. Hence for any $1\leq i\leq n/2$, we obtain $[X, v_{2i}] = [v_{2i-1}, v_{2i}] + [ v_{2i+1}, v_{2i}] = 0$ by skew symmetry of the bracket. Thus, $\n_{C_n}$ has a nontrivial abelian factor.
\end{proof}

\begin{cor}\label{Cor:CnoddNoEdR} Let $n$ be odd and  let $C_n$ be a cycle with standard orientation and at least two distinct edge labels.  Then $\n_{C_n}$ has trivial abelian factor.
\end{cor}

\begin{proof} This follows directly from Proposition \ref{prop:Cnoddpath}, as the decomposition into even paths required for a nontrivial abelian factor is impossible in a cycle with an odd number of edges. 
\end{proof}

\section{\textbf{Schreier graphs}}\label{Sec:Schreier}

The next family of graphs that we will consider are Schreier graphs, which are a generalization of Cayley graphs. In general, Schreier graphs are not simple, so multiple edges between vertices as well as loops on a vertex are allowed.

\begin{defn}\label{defn:regular} A \emph{$k$-regular graph} is a graph such that every vertex has degree $k$. We will say a graph is \emph{even-regular} if $k$ is even.
\end{defn}

\begin{defn} A \emph{Schreier graph} $G=(V,E,C)$ is an even-regular directed edge-labeled graph, where for each vertex $v\in V(G)$ and for each label $Z \in C(G)$, there is one edge starting at $v$ with edge label $Z$ and one edge terminating at $v$ with same edge label $Z$. Note that this could be the same edge, forming a loop at vertex $v$. \end{defn}

\begin{figure}[ht]
\begin{center} 
\tikzset{every loop/.style={distance=10mm, looseness=10}}
\begin{tikzpicture}[->,>=stealth',shorten >=1pt,auto,
  thick,vertex/.style={circle,draw,fill,scale=.5},node distance=2in,thick,every edge quotes/.append style = {font=\footnotesize}]

\node[vertex, label=left:{$v_1$}](X1) {};
\node[vertex, label=right:{$v_2$}](X2) [below right of=X1] {};
\node[vertex, label=right:{$v_3$}] (X3) [below of=X2]  {};
\node[vertex, label=left:{$v_5$}] (X5)[below left of=X1]  {};
\node[vertex, label=left:{$v_4$}] (X4)[below of=X5]  {};

\path[every node/.style={font=\sffamily\small}]
(X1) edge ["$Z_2$"] (X2) edge [in=340, out=60, loop,"$Z_1$"] (X1)
(X2)     edge ["$Z_1$"]  (X3)
				 edge node[pos=.5,above] {$Z_2$} (X5)
(X3) edge  [bend left,"$Z_2$"]	 (X4)	
		     edge node[pos=.2,above] {$Z_1$}   (X5)
(X4) edge [bend left]  node[pos=0.5,below] {$Z_2$} (X3)
					edge node[pos=.2,above] {$Z_1$} (X2)
(X5) edge["$Z_2$"] (X1)
				edge node[pos=0.5,left] {{\color{black}$Z_1$}} (X4)
				;
\end{tikzpicture}
\caption{Example of a Schreier graph}
\label{Fig:Schr}
\end{center}
\end{figure}
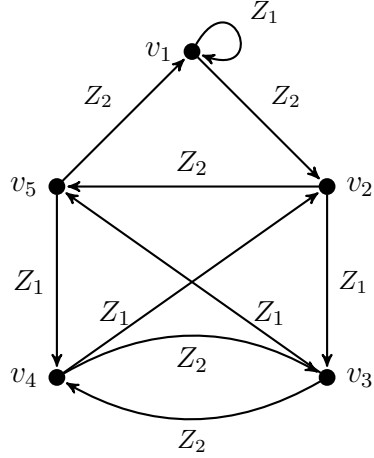

The technical definition of a Schreier graph comes from a group, subgroup, and a set of generators for the group. The vertices are the right cosets and the edges are formed by right inverse multiplication of the cosets by each of the group generators. For details on the group theoretic motivation for this construction, see \cite[Definition 2.1]{R}.  By \cite[Theorem 2]{Gr}, every connected even-regular graph, is a Schreier graph for a certain selection of group, subgroup, and generating set. This paper is focused on the graph theoretic  structure of the Schreier graphs rather than which groups and subgroups create these graphs.

One important property of a Schreier graph $G$ that we are utilizing is a group action on the vertices of the graph. This group action $\alpha$ on the vertices of $G$ is defined such that $$\alpha(Z)v=w \mbox{ if and only if $(v,w) \in E(G)$ with edge-label $Z$.}$$ This group action is then expanded to the entire group by writing each element in the group as a product of the generators in the edge-label set. We have $$\alpha(Z_1Z_2)v=\alpha(Z_1)(\alpha(Z_2)v).$$ From this we can see that $$\alpha(Z^{-1})w=v \mbox{ if and only if } \alpha(Z)v=w.$$

\begin{ex}\label{schrhouse} Given the Schreier graph in Figure~\ref{Fig:Schr}, we have the following:
\begin{eqnarray*} \alpha(Z_1)v_1 &=& v_1 \mbox{  since there is a $Z_1$ loop from $v_1$ to $v_1$,} \\ \alpha(Z_2)v_1 &=& v_2 \mbox{ since there is a $Z_2$ edge from $v_1$ to $v_2$,}\\ \alpha(Z_2^{-1})v_1 &=& v_5 \mbox{ since there is $Z_2$ edge from $v_5$ to $v_1$, and}\\ \alpha(Z_1^2)v_2 &=& v_5 \mbox{ since there are $Z_1$ edges from $v_2$ to $v_3$ and from $v_3$ to $v_5$.}\end{eqnarray*}
\end{ex}

\begin{prop}\label{defn:jz} (\cite{R} Theorem 4.2) Let $G$ be a Schreier graph and $\n_G$ the associated Lie algebra constructed by Definition \ref{defn:construction}. Then the $j$-operator on $\n_G$ is given by $$j(Z)v=\alpha(Z)v-\alpha(Z^{-1})v,$$ for all $Z$ in $[\n_G,\n_G]$ and for all $v\in \V$. \end{prop}

Given a Schreier graph $G$ and the Lie algebra $\n_G$ associated to it as constructed in Definition~\ref{defn:construction}, we wish to find a basis for the abelian factor of $\n_G$. We begin by adapting Definition~\ref{defn:scripte} of $\E$ to a sequence of paths of length two with the same label and then define an equivalence relation on the vertices of $G$ based on these sequences.

\begin{defn}\label{defn:equivalence} Let $G=(V,E,C)$ be a Schreier graph. \begin{enumerate}
\item  We define a relation $\sim$ on $V(G)$:
$v_i \sim v_j$  if there exists a sequence of edge labels $\{Z_1, \dots, Z_{\ell}\} \subseteq C$ (edge labels not necessarily distinct) such that \[v_j=\alpha\Big((Z_{\ell}^{\pm 1})^2 \cdots (Z_2^{\pm 1})^2 (Z_1^{\pm 1})^2\Big)v_i.\]
\item We call this product $(Z_{\ell}^{\pm 1})^2 \cdots (Z_2^{\pm 1})^2 (Z_1^{\pm 1})^2$ a {\em 2-path sequence}. We will denote the set of all 2-path sequences by $\mathcal{T}$. In other words $v_i \sim v_j$ if and only if there exists a 2-path sequence $T \in \mathcal{T}$ such that $v_j = \alpha(T) v_i,$ i.e., if there exists a 2-path sequence connecting $v_i$ to $v_j$.
\end{enumerate} \end{defn}

\begin{ex}
In Figure~\ref{Fig:Schr}, $v_1 \sim v_2$ because $v_2 = \alpha(Z_2^2Z_2^2)(v_1)$. We also note that $v_2 = \alpha((Z_2^{-1})^2) v_1$. Hence the 2-path sequence connecting the vertices  is not necessarily unique.
\end{ex}

\begin{lemma} \label{equivalence} Let $G=(V,E,C)$ be a Schreier graph. Then the relation $\sim$ from Definition~\ref{defn:equivalence} is an equivalence relation.
\end{lemma}

\begin{proof} (Reflexive): Since $G$ is a Schreier graph, for each vertex $v_i \in V$, there is an edge labeled $Z$ starting at $v_i$ and terminating at some vertex $v_j$, for some $Z \in C$. Vertex $v_j$ will also have an edge labeled $Z$ starting at $v_j$ and ending at some vertex $v_k$ (Note that these vertices and edges may not be unique). 
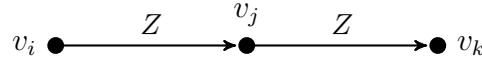
\begin{center} \begin{tikzpicture}[->,>=stealth',shorten >=1pt,auto,
  thick,vertex/.style={circle,draw,fill,scale=.5,font=\sffamily\large\bfseries},node distance=2in,thick]

\node[vertex, label=left:{$v_i$}](vi) {};
  \node[vertex, label=above:{$v_j$}](vj) [right of=vi] {};
  \node[vertex, label=right:{$v_k$}](vk) [right of=vj]  {}; 
  \path[every node/.style={font=\sffamily\small}]
    (vi) edge node  {$Z$}  (vj)
		(vj)     edge  node  {$Z$}  (vk)
				;
\end{tikzpicture} \end{center}

Consider $T=(Z^{-1})^2 Z^2 \in\mathcal{T}$. Then 
$\alpha(T)v_i =v_i$. Thus $v_i \sim v_i$.

(Symmetric): Assume that $v_i \sim v_j$. Therefore, there exists a 2-path sequence $T=(Z_{\ell}^{\pm 1})^2 \cdots (Z_2^{\pm 1})^2 (Z_1^{\pm 1})^2$ such that $v_j=\alpha\Big((Z_{\ell}^{\pm 1})^2 \cdots (Z_2^{\pm 1})^2 (Z_1^{\pm 1})^2\Big)v_i$. It follows that $v_i=\alpha\Big((Z_{1}^{\mp 1})^2 (Z_2^{\mp 1})^2 \cdots (Z_{\ell}^{\mp 1})^2\Big)v_j$, where the inverse of $(Z_i)^{\pm}$ is $Z_i^{\mp}$. Therefore $v_j \sim v_i$.\\

(Transitive): Assume that $v_i \sim v_j$ and $v_j \sim v_k$. Then there exists 2-path sequences $T_1,T_2\in\mathcal{T}$ such that $v_j=\alpha(T_1)v_i \mbox{  and } v_k=\alpha(T_2)v_j.$ Thus, $v_k=\alpha(T_2)\Big(\alpha(T_1)v_i\Big)=\alpha(T_2T_1)v_i$. Since $T_2T_1\in\mathcal{T}$, $v_i \sim v_k$.\\
Therefore $\sim$ is an equivalence relation. \end{proof}

Next, we take  the partition of $V(G)$ into equivalence classes with respect to the equivalence relation $\sim$ from Definition~\ref{defn:equivalence}. We will denote the equivalence class containing $v_i$ by $[v_i]$. 

For the Schreier graph in Figure~\ref{Fig:Schr} above, the equivalence relation would partition the vertices into two equivalence classes: $[v_1]=\{v_1, v_2, v_5\}$ and $[v_3]=\{v_3, v_4\}$. 
The sum of the vertices in each equivalence class will form basis elements for the abelian factor of the associated Lie algebra as detailed below.

\begin{thm}\label{edrbasis} Let $\n_G$ be the Lie algebra associated with the Schreier graph $G=(V,E,C)$ as defined above. Let $\beta$ be the the number of distinct equivalence classes in the partition of $V(G)$ with respect to $\sim$ from Definition~\ref{defn:equivalence}. For ease of notation, we will relabel vertices if necessary in order that the vertices $v_1, v_2, \ldots, v_\beta$ represent the distinct equivalence classes. Define $\ds \xi_i=\sum_{v_{j}\in[v_i]} v_{j}$, for each $i$. Then the set $\{\xi_i\}_{i=1}^{\beta}$ forms a basis for the abelian factor of $\n_G$. \end{thm}

The following lemmas will lead us to the proof of this theorem.

\begin{lemma}\label{PermEq}  For each $Z\in C$, $\alpha(Z^2)$ permutes the vertices in each equivalence class $[v]$. Consequently, $\alpha(Z^2) \xi_i = \xi_i$ for all $i = 1, \ldots, \beta$ and $Z \in C$. \end{lemma}

\begin{proof}
Let $Z \in C$ and $v \in V(G)$. Let $w \in [v]$. Then there exists a 2-path sequence $T$ such that $w = \alpha(T) v$. Therefore $\alpha(Z^2) w = \alpha(Z^2)(\alpha(T) v) = \alpha(Z^2 T)v \in [v]$ as $Z^2 T$ is also a 2-path sequence.  Also, if $w\in[v]$, then $\alpha(Z^{-1})^2 w\in[v]$ since $(Z^{-1})^2 \in \mathcal{T}$. Then we see that $\alpha(Z^2)\Big(\alpha(Z^{-1})^2 w\Big)=w$.

Now,
\begin{eqnarray*} \alpha(Z^2)\xi_i &=& \alpha(Z^2) \sum_{v_j \in [v_i]} v_j\\
&=&  \sum_{v_j \in [v_i]}  \alpha(Z^2) v_j\\
&=& \sum_{v_k \in [v_i]} v_k, \end{eqnarray*}

because $\alpha(Z^2)$ permutes the vertices in $[v_i]$. Thus $\alpha(Z^2)\xi_i = \xi_i$ for each $i = 1, \ldots, \beta $.
\end{proof}

\begin{lemma}\label{xiabelian}For each $i=1, \ldots, \beta$, $\xi_i$ is in the abelian factor of $\n_G$. \end{lemma}

\begin{proof} By Lemma~\ref{PermEq}, we know $\alpha(Z^2)\xi_i=\xi_i$ for each $i$ and for all $Z \in C$. By applying $\alpha(Z^{-1})$ to both sides of this equation, we get $\alpha(Z)\xi_i = \alpha(Z^{-1})\xi_i$, which means $\alpha(Z)\xi_i-\alpha(Z^{-1})\xi_i =0$.  By Proposition \ref{defn:jz}, $j(Z)\xi_i=0$ for each $i$ and for all $Z \in C$. Hence by Lemma \ref{Lemma:AbFactor}, $\xi_i$ is in the abelian factor for all $i=1,\dots,\beta$. \end{proof}

\begin{lemma}\label{MainLemma} Let $X$ be in the  abelian factor of $\n_G$. Then $\alpha(T) X = X$ for all 2-path sequences $T \in \mathcal{T}$. Consequently, $X$ is in the span of $\{\xi_i\}_{i=1}^\beta$, using the notation of Theorem~\ref{edrbasis}. 
\end{lemma}

\begin{proof}
Let $X$ be in the abelian factor of $\n_G$. Then by Lemma \ref{Lemma:AbFactor}, it implies that $j(Z)X = 0$ for all $Z \in C$. By Proposition~\ref{defn:jz}, this means that $\alpha(Z)X-\alpha(Z^{-1})X = 0$. Hence $\alpha(Z) X = \alpha(Z^{-1}) X$ and $\alpha((Z^{\pm 1})^2) X = X$ for all $Z \in C$. By induction, we get  $\alpha(T) X = X$ for all 2-path sequences $T \in \mathcal{T}$.

Let the vertex set $V = \{v_1, \ldots, v_n\}$. Assume  that  $X =  \sum_{k=1}^{n} b_k v_k$ is in the abelian factor of $\n_G$. We will prove that $b_i = b_j$ whenever $v_i \sim v_j$.  Suppose that $v_i \sim v_j$ for some $i$ and $j$. By definition of $\sim$, there exists a 2-path sequence $T$ such that $\alpha(T)(v_i) = v_j$. We just proved that $X$ is invariant under $\alpha(T)$ and hence 
\[
   \sum_{k=1}^n b_k \alpha(T) v_k =  \sum_{k=1}^{n} b_k v_k.
\]

Since $\alpha(T)(v_i) = v_j$, we can rewrite the above equation as follows:

\[b_i v_j + \sum_{k \neq i} b_k \alpha(T) v_k =  b_j v_j + \sum_{k \neq j} b_k v_k.\]

Using the fact that $\alpha(T)$ permutes the set of all vertices $\{v_k\}$, there does not exist $k \neq i$ such that $\alpha(T) v_k = v_j$. Also the set $\{v_k\}$
is linearly independent. Hence $b_i = b_j$. Therefore, we have proved that if $v_j$ is in the equivalence class $[v_i]$, then $b_j = b_i$. Thus \[ \ds X = \sum_{i=1}^{\beta} \sum_{v_j \in [v_i]} b_i v_j = \sum_{i=1}^{\beta} b_i \sum_{v_j \in [v_i]} v_j = \sum_{i=1}^{\beta} b_i \xi_i. \]
This proves that $X$ is in the span of $\{\xi_i\}_{i=1}^{\beta}$.
\end{proof}

\noindent {\it Proof of Theorem \ref{edrbasis}.}    Let the vertex set $V = \{v_1, \ldots, v_n\}$. Since $V$ is a linearly independent set and $\{[v_i]\}_{i=1}^{\beta}$ is a partition of $V(G)$, it follows directly that $\{\xi_i\}_{i=1}^{\beta}$ is a linearly independent set. Also the set  $\{\xi_i\}_{i=1}^{\beta}$ is a subset  of the abelian factor by Lemma~\ref{xiabelian}. Moreover, by Lemma~\ref{MainLemma}, the abelian factor is the span of $\{\xi_i\}_{i=1}^{\beta}$ . Therefore $\{\xi_i\}_{i=1}^{\beta}$ forms a basis for the abelian factor of $\n_G$. \qed

\begin{ex} For the Schreier graph $G$ in Figure~\ref{Fig:Schr}, the resulting Lie algebra $\n$ would have a two-dimensional abelian factor with basis $\{v_1+v_2+v_5,v_3+v_4\}$.
\end{ex}

\begin{ex} Consider the cycle $C_n$ with standard orientation and all edge labels the same. Because $C_n$ can be considered a Schreier graph, we can use Theorem~\ref{edrbasis} to compute a basis for the abelian factor of $\n_{C_n}$.
\begin{enumerate}
\item If $n=2k$ is even, there will be two distinct equivalence classes: $[v_1]=\{v_{2i-1}\}_{i=1}^{k}$ and $[v_2] = \{v_{2i}\}_{i=1}^k$.
\item If $n$ is odd, there will be one equivalence class $[v_1]=\{v_i\}_{i=1}^{n}$.
\end{enumerate}
This will result in the same basis for the abelian factor of $\n_{C_n}$ as in Theorem~\ref{Prop:Cnstandard}.
\end{ex}

\section{\textbf{Graph coloring}}\label{Sec:coloring}

One area of interest in graph theory is edge-colored graphs. This is equivalent to edge-labeled graphs by defining the set of edge labels to be the set of colors. We can then use results from Section~\ref{Preliminaries} to study the associated 2-step nilpotent Lie algebras. We assume that all graphs in this section are simple.

\subsection{\textbf{\textit{Proper edge coloring}}}
\begin{defn}\label{defn:proper}
An edge coloring on a graph $G$ is called a {\em proper edge coloring} if no two incident edges have the same color. 
\end{defn}

\begin{prop} Let $G$ be a simple graph with proper edge coloring. Then the associated 2-step nilpotent Lie algebra $\n_G$ has no abelian factor. 
\end{prop}

\begin{proof} Since no edge label is repeated among the edges incident at any single vertex,  $\E = \emptyset$. Thus, by Corollary \ref{cor:nullscriptE}, $\n_G$ has no abelian factor.
\end{proof}

\begin{defn} For a graph $G$, the \emph{chromatic index} of $G$ is denoted $\chi'(G)$ and is the minimum number of colors required for a proper edge coloring of $G$.
\end{defn}

Observe that if $G$ has a proper edge coloring using $k$ colors, then  a proper edge coloring using more colors, up to the number of edges, can be obtained from a proper $k$ coloring by introducing a new color on one edge whose label appears more than one time in $G$.   

Recall the following well-known result from Graph Theory, see \cite{BM} for example. Denote the maximum degree of a vertex in $G$ by $\Delta(G)$.  

\begin{thm}[Vizing's Theorem]  For any finite simple graph $G$, \[ \Delta(G) \leq \chi'(G) \leq \Delta(G) +1.\]
\end{thm}

If $G$ is bipartite, then $\chi'(G) = \Delta(G)$. If $G$ is a regular graph (recall Definition~\ref{defn:regular}) on an odd number of vertices, then $\chi'(G) = \Delta(G)+1$ \cite{BM}. Generally speaking,  Paul Erd\H{o}s proved that almost all graphs have $\chi'(G) = \Delta(G)$ \cite{Er}. 

\begin{prop}
Let $C_n$ be a cycle on $n$ vertices where $n$ is even. Then there exists an edge labeling so that $\n_{C_n}$ has no abelian factor and $dim([\n_{C_n},\n_{C_n}])=k$ for each $2\leq k \leq n$.
\end{prop}

\begin{proof} For $n$ even, the cycle $C_n$ is bipartite, so, $\chi'({C_n}) = \Delta({C_n}) =2$. Thus a proper edge coloring exists with $k$ colors where $2\leq k \leq n$.
\end{proof}

For every simple connected directed graph $G$, there exists an edge coloring of $G$ that is proper. Thus, for every graph with $\chi'(G) < |E(G)|$, there is an edge coloring with multiple edges having the same label that the resulting 2-step nilpotent Lie algebra $\n_G$ has no abelian factor. In fact, for each $k$ with $\chi'(G) \leq k \leq |E(G)|$, we can find an edge coloring of $G$ so that the resulting Lie algebra $\n_G$ has no abelian factor and has $dim ([\n_G,\n_G]) =k$. 

These results are consistent with the results of sections \ref{Sec:stars} and \ref{Sec:cycles}. We illustrate with the following examples:

\begin{ex}
\begin{enumerate}
\item
Let $K_{1,n}$ be a star graph. Then $\chi'(K_{1,n})=|E(K_{1,n})| = n$, so there is only one proper coloring of $K_{1,n}$. If we color the edges so that the coloring is not proper, then by Theorem \ref{Thm:DimEdRStar} the corresponding Lie algebra will have a nontrivial abelian factor. 

\item Let $C_{2n}$ be a cycle on an even number of vertices. We have $\chi'({C_n}) = 2$, and any proper coloring by $k$ colors for $2\leq k\leq 2n$ will result in $\n_{C_n}$ with no abelian factor. However, there are non-proper edge colorings which also result in $\n_{C_n}$ with trivial abelian factor. See Proposition \ref{prop:Cnoddpath} for a precise description. 
\end{enumerate}
\end{ex}

\subsection{\textbf{\textit{An application to uniformly colored graphs}}}
We now consider a family of properly edge-colored graphs known as uniformly colored graphs. 
In \cite{PS}, the authors give a correspondence between uniform Lie algebras and the 2-step nilpotent Lie algebras associated with uniformly colored graphs. 

\begin{defn}\label{defn:uniform}Let $G$ be a a regular graph where every vertex has degree $s>0$. A proper edge coloring of the graph $G$ is called a \emph{uniform coloring} using $p$ colors if each label is used $r$ times.  
 If $G$ has $q$ vertices and if each label is used $r$ times, then $G$ is called a $(p,q,r)$ uniform graph. 
\end{defn}

Since a uniform coloring is a proper coloring, we have that a graph with a uniform coloring also results in the associated Lie algebra $\n_G$ having no abelian factor. Hence the derived algebra of $\n_G$ will be equal to the center of $\n_G$. 

We now give an application to the singularity properties of Lie algebras arising from a graph $G$ with a uniform edge coloring. Refer back to Definition~\ref{defn:singular} for the definition of singularity. 

\begin{prop}\label{Prop:uniformPSsingANS} Let $G$ be a simple $(p,q,r)$ uniformly colored graph with an even number of vertices such that the degree of each vertex, $s$, coincides with the number of edge colors, $p$. Then the associated 2-step nilpotent Lie algebra $\n_G$ is either nonsingular or almost nonsingular. 
\end{prop}

\begin{proof}
Let $G$ be uniformly colored by edge labels $\{Z_1, \ldots, Z_s\}$. Then each edge label appears $r$ times in $G$ where $r=2q$. Consider the subgraph of $G$ which is induced on the edges labeled $Z_i$, and denote this subgraph by $G_i$.   Observe   since $G$ is uniformly colored and since each vertex has precisely one incident edge with each edge label,  the graph $G_i$ forms a vertex covering of $G$ by disjoint copies of $K_2$. For each $i$, $1 \leq i \leq s$, there exists a permutation $\sigma_i$ of the set $\{1, \ldots, q=2r\}$ so that in the ordered basis $\{X_{\sigma_i(k)}\}$ for $\V \subset \n_G$, the map $j(Z_i)$ has matrix of size $2r \times 2r$ in block diagonal form which has $r$ copies of the $2\times 2$ block $\bigl( \begin{smallmatrix} 
  0 & -1\\
  1 & 0 
\end{smallmatrix} \bigr)$ on the diagonal. It is now apparent that $j(Z_i)$ has eigenvalues $\pm i$. Thus, $j(Z_i)$ is nonsingular for all $1 \leq i \leq s$. Hence, the Lie algebra $\n_G$ is either nonsingular or almost nonsingular. 
\end{proof}

\begin{ex}
Let $G$ be the $(3,6,3)$ uniformly colored graph in Figure~\ref{Fig:uniform}. 

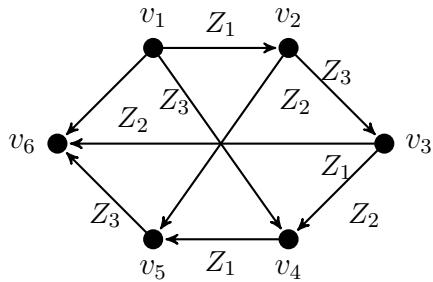
\begin{figure}[ht]
\begin{center}
\begin{tikzpicture}[->,>=stealth',shorten >=1pt,auto,node distance=3cm,
  thick,main node/.style={circle,fill,scale=.6,draw,font=\sffamily\Large\bfseries}]

  \node[main node] (v1) [label=above:{$v_1$}] {};
  \node[main node] (v2) [right of=v1] [label=above:{$v_2$}] {};
  \node[main node] (v3) [below right of=v2] [label=right:{$v_3$}] {};
	\node[main node] (v4) [below left of=v3] [label=below:{$v_4$}] {};
	\node[main node] (v5) [left of=v4] [label=below:{$v_5$}] {};
	\node[main node] (v6) [below left of =v1] [label=left:{$v_6$}] {};
	
  \path[every node/.style={font=\sffamily\small}]
    (v1) edge  node [above] {$Z_1$}  (v2)
		(v1) edge node [below, very near start] {$Z_3$} (v4)
		(v1) edge  node {$Z_2$} (v6)
		(v2) edge node [above] {$Z_3$}  (v3)
		(v2) edge  node [very near start] {$Z_2$} (v5)
		(v3)   edge  node   {$Z_2$}  (v4)
		(v3) edge node[very near start]  {$Z_1$}  (v6)
		(v4) edge  node{$Z_1$} (v5)
			(v5) edge  node [below] {$Z_3$}  (v6);
      \end{tikzpicture}
\end{center}
\caption{Example of a $(3,6,3)$ uniformly colored graph}
\label{Fig:uniform}
\end{figure}

Then a direct computation shows that for $Z_0 = aZ_1 + bZ_2 +cZ_3$, the map $j(Z_0)$ has the following matrix in the basis for $\V \subset \n_G$ given by the vertices of $G$. 
 \[ \left( \begin{matrix}
0& -a & 0 & -c & 0 & -b\\
a & 0 & -c & 0 & -b & 0 \\
0 & c & 0 & -b & 0 & -a\\
c & 0 & b & 0 &-a & 0 \\
0 & b & 0 & a & 0 & -c\\
b & 0 & a & 0 & c & 0 
\end{matrix}
\right) 
\]
 
The map $j(Z_0)$ is nonsingular. 
\end{ex}

\begin{ex} Let $C_8$ be a cycle on 8 vertices with standard orientation and edge labels $Z_1, Z_2, Z_3, Z_4$ repeated in that order around the cycle twice. This graph is a  $(4, 8, 2)$ uniformly colored graph. For each $1 \leq i \leq 4$ we have $j(Z_i)$ singular. This is evident from the graph because each edge is not incident on 4 of the vertices, giving each $j(Z_i)$ a nontrivial nullspace.  However,  a direct computation shows $j(Z_1+Z_3)$ is nonsingular. Hence the resulting Lie algebra $\n_{C_n}$ is almost nonsingular. 
\end{ex}

\section{Conclusion}

This article provides several results in the study of how  edge-labeling patterns in the graph can affect the algebraic structure of the resulting Lie algebra. These results are primarily on the existence and description of an abelian factor and singularity of nilpotent Lie algebras. While the results stand on their own, there are several open questions that the authors suggest for further research.  \\

\noindent{\bf{Open Question 1. }} Is there a necessary and sufficient condition on a labeled graph $G$ to determine the existence of a nontrivial abelian factor in $\n_G$? Currently no such condition has been found, even in the case where $G$ is a simple graph, leaving much left to investigate. We have shown that changing the label or direction of a single edge can alter the structure of the Lie algebra and its abelian factor. From the results presented here, it is evident that the same-labeled paths of length two and their incidence in the graph $G$ will play a significant  role in determining the algebraic properties of $\n_G$.\\

\noindent {\bf{Open Question 2}}. For a given Lie algebra $\n$, can either the graph labeling or construction method of the Lie algebra be adapted in such a way that would result in an isomorphic Lie algebra $\n'$ that could be used more effectively to study these Lie algebras further? The results in this article suggest that the study of Lie algebras associated to graphs using different or adapted constructions may prove fruitful. For star graphs, we describe a weighted graph with fewer edges that still encoded the same information about the Lie algebra (Example~\ref{ex:StarG1}). For any Schreier graph $G$, it is possible to construct a new graph whose vertices are the equivalence classes from Definition~\ref{defn:equivalence}. This construction is useful for determining the abelian factor of $\n_G$, but does not immediately generalize to non-Schreier graphs. \\

\noindent {\bf {Open Question 3.}} Given that a labeled graph gives rise to a natural inner product on the associated Lie algebra $\n_G$, 
to what extent are the geometric properties of these constructed Lie algebras determined by the graph theoretic structure? We discussed singularity in Sections~\ref{Sec:stars} and \ref{Sec:coloring}. A particular topic of interest to the authors is which Lie algebras admit metrics that are Heisenberg-like. In \cite[Corollary 4.3]{DDM}, the first three authors found that the Lie algebras constructed from simple graphs would be Heisenberg-like if and only if the graph is either a star graph or a complete graph on three vertices. It is unknown if a similar result can be obtained when the graphs are allowed to have repeated edge labels. In the case of repeated edge labels, it may be desirable to take the Lie algebra quotient with the abelian factor. 
One obstacle to further geometric computation, including geodesic properties, is finding a suitable basis so the computation of the eigenvectors of the $j(Z)$ map is accessible.



\begin{thebibliography}{Namehere}

\bibitem{AAA} A. Alfaro,  M. Alvarez, Y.  Anza, Degenerations of graph Lie algebras,  {\it  Linear Multilinear Algebra.}  70 (1): 91-100 (2022).

\bibitem{AD} A. Andrada, I. Dotti, Killing-Yano 2-forms on 2-step nilpotent Lie groups,  {\it Geom. Dedicata}.  212: 415-424 (2021).

\bibitem{BM} J.A. Bondy, U.S.R. Murty, {\it Graph Theory.} Springer (2008).

\bibitem{CMS} D. Chakrabarti, M. Mainkar, S. Swiatlowski,  Automorphism groups of nilpotent Lie algebras associated to certain graphs,  {\it Comm. Algebra.} 48 (1): 263-273  (2020).

\bibitem{CDF} D. Conti,   V. del Barco, F. Rossi,  Diagram involutions and homogeneous Ricci-flat metrics,  {\it Manuscripta Math.} 165 (3-4): 381-413  (2021). 

\bibitem{DM} S. G. Dani, M. G. Mainkar, Anosov automorphisms on compact nilmanifolds associated with graphs, {\it Trans. Amer. Math. Soc.} 357: 2235-2251  (2005).

\bibitem{DDM} R. DeCoste, L. DeMeyer, M. G. Mainkar,  Graphs and metric 2-step nilpotent Lie algebras, {\it Adv. Geom.} 18 (3): 265-284 (2018).

\bibitem{DeMa} J. Der\'{e}, M. Mainkar, Anosov diffeomorphisms on infra-nilmanifolds associated to graphs, {\it to appear in Math. Nachr.}, https://arxiv.org/pdf/2008.09717.pdf (2021).

\bibitem{E} P.  Eberlein, Geometry of 2-step nilpotent groups with a left invariant metric, I, {\it Ann.  Scient. \'Ecole Normale Sup. (4)} 27 (5): 611-660 (1994). 
		
\bibitem{E2} P. Eberlein, Geometry of 2-step nilpotent groups with a left invariant metric II, {\it Trans. Amer. Math. Soc.} 343: 805-828 (1994). 

\bibitem{Er} P. Erd\H{o}s, On the chromatic index of almost all graphs, {\it J. Combin. Theory Ser. B.} 23: 225-257 (1977).

\bibitem{F} H. Fana\H{i}, Einstein solvmanifolds and graphs,  {\it C. R. Math.} 344 (1): 37-39 (2007).

\bibitem{FJ} M. Farinati, A. Jancsa, Lie bialgebra structures on 2-step nilpotent graph algebras, {\it  J. Algebra.} 505: 70-91, (2018).

\bibitem{GM} R. Gornet, M. Mast, The length spectrum of Riemannian two-step nilmanifolds, {\it Ann.  Scient. \'Ecole Normale Sup. (4)} 33 (2): 181-209 (2000).   

\bibitem{GGI} G. Grantcharov, V. Grantcharov, P. Iliev, Solvable Lie algebras and graphs, {\it J. Algebra.} 491: 474-489 (2017).

\bibitem{Gr}  J.L. Gross, Every connected regular graph of even degree is a Schreier coset graph, {\it J.  Combinatorial Theory.} 22: 227-232 (1977).

\bibitem{K1} A. Kaplan, Riemannian nilmanifolds attached to Clifford modules, {\it Geom. Dedicata.} 11: 127-136 (1981).

\bibitem{K2} A. Kaplan, On the geometry of groups of Heisenberg type,  {\it Bull. London Math. Soc.} 15 (1): 35-42  (1983).

\bibitem{LW1} J. Lauret,  C. Will, On Anosov automorphisms of nilmanifolds, {\it J. Pure Appl. Algebra.} 212 (7): 1747-1755 (2008).

\bibitem{LW2} J. Lauret, C.  Will, {Einstein solvmanifolds: Existence and non-existence questions,  \it Math. Ann.} 350 (1): 199-225 (2011).

\bibitem{LP} K. Lee, K. Park, Smoothly closed geodesics in 2-step nilmanifolds,  {\it Indiana Univ. Math. J.} 45: 1-14 (1996).

\bibitem{M} M. Mainkar, Graphs and two-Step nilpotent Lie algebras, {\it Groups Geom. Dyn.} 9 (1): 55-65 (2015).

\bibitem{N} Y. Nikolayevsky, Geodesic orbit and naturally reductive nilmanifolds associated with graphs, {\it Math. Nachr.} 293 (4): 754-760 (2020).

\bibitem{O} G. Ovando, The geodesic flow on nilmanifolds associated to graphs, {\it Rev. Un. Mat. Argentina.} 61 (2): 315-338 (2020). 

\bibitem{PS} T. L. Payne, M. Schroeder, Uniform Lie algebras and uniformly colored graphs, {\it Adv. Geom.} 17 (4): 507-524  (2017). 

\bibitem{PT} H. Pouseele,  P. Tirao, Compact symplectic nilmanifolds associated with graphs, {\it J. Pure Appl. Algebra.} 213 (9): 1788-1794 (2009). 

\bibitem{R} A. Ray, Two-step and three-step nilpotent Lie algebras constructed from Schreier graphs, {\it J. Lie Theory.} 26 (2): 479-495 (2016).

\end{thebibliography}
\end{document}